\documentclass{pracjourn}

\usepackage{amsmath,amsthm,amsfonts,amssymb}
\usepackage{tabularx}
\usepackage{siunitx}
\usepackage[bottom=3cm]{geometry}
\usepackage{cleveref}
\usepackage{graphicx}
\usepackage{wrapfig}

\DeclareMathOperator{\supp}{\mathop{\mathrm{supp}}}

\theoremstyle{theorem}

\newtheorem{theorem}{\textsc{Theorem}}
\newtheorem{corollary}{\textsc{Corollary}}
\newtheorem{remark}{\textsc{Remark}}
\newtheorem{lemma}{\textsc{Lemma}}
\newtheorem{proposition}{\textsc{Proposition}}
\theoremstyle{definition}

\usepackage[displaymath]{lineno}

\newcommand*\patchAmsMathEnvironmentForLineno[1]{%
   \expandafter\let\csname old#1\expandafter\endcsname\csname #1\endcsname
   \expandafter\let\csname oldend#1\expandafter\endcsname\csname end#1\endcsname
   \renewenvironment{#1}%
      {\linenomath\csname old#1\endcsname}%
      {\csname oldend#1\endcsname\endlinenomath}}%
\newcommand*\patchBothAmsMathEnvironmentsForLineno[1]{%
   \patchAmsMathEnvironmentForLineno{#1}%
   \patchAmsMathEnvironmentForLineno{#1*}}%
\AtBeginDocument{%
\patchBothAmsMathEnvironmentsForLineno{equation}%
\patchBothAmsMathEnvironmentsForLineno{align}%
\patchBothAmsMathEnvironmentsForLineno{flalign}%
\patchBothAmsMathEnvironmentsForLineno{alignat}%
\patchBothAmsMathEnvironmentsForLineno{gather}%
\patchBothAmsMathEnvironmentsForLineno{multline}%
}

\newcommand{\insertimg}[2]{
\begin{minipage}{#1\linewidth}
   \begin{center}
   \includegraphics[width=1\linewidth]{#2} 
   \end{center}
\end{minipage}
}

\renewcommand{\email}[1]{\href{mailto:#1}{#1}}

\begin{document}
\title{Global uniqueness in a passive inverse problem of helioseismology}
\author{A. D. Agaltsov\thanks{Max-Planck-Institut f\"ur Sonnensystemforschung, 
Justus-von-Liebig-Weg 3, 37077 G\"ottingen, Germany 
  (\email{agaltsov@mps.mpg.de}).}, T. Hohage\thanks{Institut f\"ur Numerische und Angewandte Mathematik, Georg-August-Universit\"at G\"ottingen, Lotzestr. 16-18, 37083 G\"ottingen, Germany 
  (\email{hohage@math.uni-goettingen.de}) and Max-Planck-Institut f\"ur Sonnensystemforschung.}, and R. G. Novikov\thanks{CMAP, Ecole Polytechnique, CNRS, Universit\'e Paris-Saclay, 91128 Palaiseau, France; and IEPT RAS, 117997 Moscow, Russia
  (\email{novikov@cmap.polytechnique.fr}).}}

\maketitle

\begin{quote}
   We consider the inverse problem of recovering the spherically symmetric sound speed, density and attenuation in the Sun from the observations of the acoustic field randomly excited by turbulent convection. We show that observations at two heights above the photosphere and at two frequencies above the acoustic cutoff frequency uniquely determine the solar parameters. We also present numerical simulations which confirm this theoretical result.\medskip
   
\medskip
   
\end{quote}

\section{Introduction}
\subsection{Acoustic field in the Sun and its measurement}
Turbulent convection in the upper layers of the solar convection zone can reach almost sonic speeds and serves as an efficient driving mechanism for acoustic oscillations \cite{christensen2014lecture}. We consider the three-dimensional equation describing these oscillations at fixed frequency $\omega>0$ proposed in \cite{gizon2017computational}:
\begin{equation}\label{eq:acoustic}
    \begin{gathered}
    -\nabla\left(\frac{1}{\rho}\nabla\bigl(\sqrt \rho \psi_\omega \bigr) \right) - \frac{\sigma^2}{c^2 \sqrt \rho} \psi_\omega = \frac{f_\omega}{\sqrt \rho},\\
    \psi_\omega = \sqrt{\rho}c^2\nabla \cdot \xi_\omega, \quad \sigma^2 = \omega^2 + 2i\omega \gamma,
    \end{gathered}
\end{equation}
where $\xi_\omega$ is the spatial matter displacement vector, $c$ is the sound speed, $\rho$ is the density, $\gamma$ is the attenuation, $f_\omega$ is the random source field due to turbulent convection and $x \in \mathbb R^3$. In this work we consider this model under an additional spherical symmetry assumption: $c = c(|x|)$, $\rho = \rho(|x|)$, $\gamma = \gamma(|x|)$. We also assume that
\begin{equation}\label{eq:spaces}
    \text{$c \in L^\infty(\mathbb R_+)$, $c\geq c_\text{min}>0$ a.e.}, \quad \text{$\rho \in W^{2,\infty}(\mathbb R_+)$, $\rho>0$}, \quad \gamma \in L^\infty(\mathbb R_+),
\end{equation}
where $\mathbb R_+ = (0,\infty)$ and $W^{2,\infty}(\mathbb R_+)$ denotes the Sobolev space of functions defined on $\mathbb R_+$ which belong to $L^\infty(\mathbb R_+)$ together with their first two derivatives. We suppose that in the upper atmosphere $|x| \geq R_a$ the sound speed is constant, the density is exponentially decreasing (which corresponds to the adiabatic approximation, see \cite[section 5.4]{christensen2014lecture}), and there are no attenuation:
\begin{equation}\label{eq:atmospheric_values}
    c(r) = c_0, \quad \rho(r) = \rho_0\exp(-(r-R_a)/H), \quad \gamma(r) = 0, \quad r\geq R_a,
\end{equation}
where $R_a = R_\odot + h_a$, $R_\odot = 6.957 \times 10^5$ km is the solar radius, $h_a$ is the altitude at which the (conventional) interface between the lower and upper parts of the atmosphere is located, and $H$ is called density scale height. The first two assumptions of formula \eqref{eq:atmospheric_values} follow the model of \cite{fournier2017atmospheric}, which extends a standard solar model of \cite{christensen1996current} to the upper atmosphere. In this article we do not fix exact values of the above parameters, but recall that in \cite{fournier2017atmospheric} they are given by the following table:
\begin{center}
\begin{tabularx}{.95\linewidth}{r|l|l}
            & value & meaning \\ \hline
     $h_a$ & 500 km & altitude at which the interface is located \\
     $c_0$ & $6855 \; \si{m.s^{-1}}$ & sound speed in the upper solar atmosphere  \\
     $\rho_0$ & $2.886 \times 10^{-6} \; \si{kg.m^{-3}}$ & density at the interface \\
     $H$ & $ 125 \; \si{km} $ & density scale height in the upper atmosphere
\end{tabularx}
\end{center}

In reality, the Sun is surrounded by a corona whose base is located at about $h_c = 2000 \, \si{km}$ above the surface and which is highly inhomogeneous. However, it is common to neglect this complication when studying acoustic waves inside of the Sun and in the lower atmosphere, see, e.g., \cite{christensen2014lecture,fournier2017atmospheric}. The adequacy of this simplification have been theoretically justified and numerically confirmed.

The exponential decay of density in the upper atmosphere results in trapping of acoustic waves with frequencies less than the cutoff frequency $\omega_\text{ctf} = c_0/(2H)$, which is about 5.2 mHz for the Sun, and to quantisation of their admissible frequencies. Observations of these frequencies and corresponding modes at the solar surface provide a common basis for helioseismological studies, see, e.g., \cite{christensen2014lecture}. In turn, acoustic waves with frequencies above the cutoff, that is, such that
\begin{equation}\label{eq:cutoff}
     \omega > \frac{c_0}{2H},
\end{equation}
propagate into the upper atmosphere. One can expect that the simulated data (oscillation power spectrum) computed from equation \eqref{eq:acoustic} is relatively closer to observations for these higher frequency waves. The reason is that convection, granulation and supergranulation significantly contribute to the power of oscillations at lower frequencies (see \cite{gizon2010local}), but are not captured by the model. Possibility of helioseismic inversions from observations above the cutoff has not been well investigated to date. In this artile we show that observations of acoustic waves, when performed at two frequencies above the cutoff and at two different heights above the solar surface, uniquely determine the sound speed, density and attenuation in the Sun in the spherically symmetric case.

Experimental measurement of solar acoustic waves can be performed through the Doppler shifts in the absorption lines of the solar light, as it is done in the Helioseismic and Magnetic Imager (HMI) onboard the Solar Dynamics Observatory (SDO) satellite, see, e.g., the official SDO website (\url{sdo.gsfc.nasa.gov}) for more details and references. HMI observes the full solar disk in the Fe I absorption line at 6173 \r{A} continuously from April 30, 2010. It combines six localised in wavelength photographs (filtergrams) taken in a neighborhood of this spectral line with a cadence of 45 sec to compute the map of Doppler velocities (Dopplergram). Several models show that this Dopplergram can serve as a rough estimate for the line-of-sight matter displacement velocity at about 100 km above the solar surface, which is the formation height for the HMI Dopplergram, see \cite{fleck2011formation}.

HMI is the successor of the Michelson Doppler Imager (MDI), which had similar design and purpose, onboard the Solar and Heliospheric Observatory (SOHO) satellite. In contrast to HMI, MDI observed the full solar disk in the Ni I absorption line at 6768 \r{A} with a cadence of 60 sec and only for several months each year during the so-called Dynamic Runs. The formation height for MDI is about 125 km above the solar surface, see \cite{fleck2011formation}. Note that observations from HMI and MDI can be found in the Joint Science Operations Center database at Stanford University (\url{jsoc.stanford.edu}). 

In the present work we assume that the measurements of the solar acoustic field can be performed at two different heights above the surface. In a rough approximation, HMI and MDI Dopplergrams taken during the Dynamics Run 2010, when both instruments continuously observed the full solar disk, can be used to extract this data. As recently shown in \cite{nagashima2014interpreting}, it is also possible to perform multi-height measurements by combining six raw HMI filtergrams in different ways.

The main theoretical results of this work are presented in \cref{sec:main_results}. Related proofs are given in \cref{sec:demonstration_of_uniqueness}. Numerical reconstructions confirming our theoretical conclusions are given in \cref{sec:reconstruction}.

\section{Main results}\label{sec:main_results}

\subsection{Extracting the imaginary part of the Green's function}\label{sec:extracting_green_function}
Under the assumptions \eqref{eq:spaces}, \eqref{eq:atmospheric_values}, \eqref{eq:cutoff} equation \eqref{eq:acoustic} at fixed $\omega$ can be rewritten as the Schr\"odinger equation
\begin{equation}\label{eq:schroedinger}
    (L_v - k^2)\psi = f, \quad L_v =  -\Delta + v, \quad k > 0,
\end{equation}
where the indices indicating dependence on $\omega$ are suppressed,
\begin{equation}\label{eq:acoustic_kv}
    k = \sqrt{\frac{\omega^2}{c_0^2} - \frac{1}{4H^2}}, \quad v(x) = k^2 - \frac{\sigma(|x|)^2}{c(|x|)^2} + \rho(|x|)^{\frac 1 2} \Delta(\rho(|x|)^{-\frac 1 2}),
\end{equation}
$v \in L^\infty(\mathbb R^3)$, and $v(x)=1/(H|x|)$ for $|x|\geq R_a$. 


In this article we consider equation \eqref{eq:schroedinger} with a general complex-valued potential $v$ such that 
\begin{equation}\label{eq:potential}
    \begin{gathered}
    v(x) = \widetilde v(|x|), \; x \in \mathbb R^3, \\
    \widetilde v \in L^\infty(\mathbb R_+), \quad \widetilde v(r) = \frac{\alpha}{r}, \; r \geq R_a, \\
    \text{for some constants $\alpha \in \mathbb R$, $R_a > 0$}.
    \end{gathered}
\end{equation}

If the potential $v$ satisfies \eqref{eq:potential}, then the resolvent $(L_v-k^2)^{-1}$ is a meromorphic operator-valued function of $k\in\mathbb C_+ = \bigl\{ z \in \mathbb C \colon \Im z > 0\bigr\}$ with the distributional kernel $G_v(x,x') = G_v(x,x';k)$ admitting a unique meromorphic continuation across the positive real axis. The restriction to $k \in \mathbb R_+$ of the distributional kernel $G_v(x,x')$ is called the radiation Green's function for equation \eqref{eq:schroedinger}. In addition, $G_v(x,x')$ is a distributional solution to equation $(L_v - k^2)G_v(\cdot,x') = \delta_{x'}$, where $\delta_{x'}$ denotes the Dirac delta function centered at $x'$. We also suppose that $k \in \mathbb R_+ \setminus\Sigma^P_v$, where
\begin{equation}\label{eq:Sigma_definition}
   \begin{gathered}
   \text{$\Sigma^P_v \subset \mathbb R_+$ is the union of the sets of positive poles $k$}\\
 \text{of functions $G_v(x,x';k)$ and $G_{\overline v}(x,x';k)$}.
  \end{gathered}
\end{equation}

Basic properties of $G_v$ can be found in \cite{saito1974principle,agmon1992analyticity}. In particular, at fixed $k$ the function $G_v$ is jointly continuous outside the diagonal $\Delta = \{ (x,x')\in \mathbb R^3 \times \mathbb R^3 \colon x = x'\}$. Besides, $(L_v-(k+i0)^2)^{-1} \in \mathcal L(L^2_{1+\varepsilon}(\mathbb R^3),L^2_{-1-\varepsilon}(\mathbb R^3))$, $\varepsilon \in (0,\tfrac 1 2]$, where $L^2_\delta(\mathbb R^3)$ denotes the Hilbert space of measurable functions $u$  in $\mathbb R^3$ with the finite norm
\begin{equation*}
    \|u\|_{L^2_\delta} =  \left[\int_{\mathbb R^d} (1+|x|)^{\delta}|u(x)|^2 \, dx\right]^{1/2}, \quad \delta \in \mathbb R.
\end{equation*}
Accordingly, for any $k \in \mathbb R_+ \setminus \Sigma^P_v$ equation \eqref{eq:schroedinger} with $f \in L^2_{1+\varepsilon}(\mathbb R^3)$, $\varepsilon \in (0,\tfrac 1 2)$, admits a unique radiation (limiting absorption), solution given by 
\begin{equation}\label{eq:radiation_solution}
    \psi_v(x) = \int_{\mathbb R^3} G_v(x,x')f(x') \, dx'.
\end{equation}

Spherical symmetry of the potential $v$ allows to separate variables in equation \eqref{eq:schroedinger}, reducing it to an equivalent multi-channel Schr\"odinger equation on the half-line $\mathbb R_+$ with non-coupled channels. More precisely, consider the orthogonal expansions in normalized spherical harmonics $Y^m_\ell$:
\begin{equation}\label{eq:harmonics_expansions}
  \psi_v(r\vartheta) = \frac 1 r \sum_{\ell \geq 0}\sum_{|m|\leq \ell} \varphi^m_{v,\ell}(r) Y^m_\ell(\vartheta), \quad f(r\vartheta) = \frac 1 r \sum_{\ell \geq 0}\sum_{|m|\leq \ell} f^m_\ell(r) Y^m_\ell(\vartheta),
\end{equation}
where $r>0$, $\vartheta \in S^2_1$ and $S^2_R = \bigl\{ x \in \mathbb R^3 \colon |x| = R\}$. Plugging these expansions into formula \eqref{eq:schroedinger}, we get the radial equations
\begin{equation}\label{eq:radial_schroedinger_0}
    (L_{v,\ell} - k^2)\varphi^m_{v,\ell} = f^m_\ell, \quad L_{v,\ell} = -\tfrac{d^2}{dr^2} + \tfrac{\ell(\ell+1)}{r^2} + \widetilde v,
\end{equation}
where $\ell \geq 0$, $|m| \leq \ell$. Besides, it follows from formulas \eqref{eq:radiation_solution}, \eqref{eq:harmonics_expansions} that if $\psi_v$ is a unique radiation solution of equation \eqref{eq:schroedinger}, then $\varphi^m_\ell$ can be expressed as:
\begin{equation}\label{eq:radial_radiation_solution}
    \varphi^m_{v,\ell}(r) = \int_0^{R_a} G_{v,\ell}(r,r')f^m_\ell(r') \, dr',
\end{equation}
where $G_{v,\ell}(r,r')$ is the coefficient in the spherical harmonics expansion of $G_v$:
\begin{equation}\label{eq:green_harmonics_expansion}
    G_v(r\vartheta,r'\vartheta) = \frac{1}{rr'}\sum_{\ell\geq 0} \sum_{|m|\leq \ell} G_{v,\ell}(r,r') Y^m_\ell(\vartheta)\overline{Y^m_\ell}(\vartheta'), \quad r,r'>0, \; \vartheta,\vartheta'\in S^2_1.
\end{equation}
One can show that $G_{v,\ell}$ is indeed a Green's function for equation \eqref{eq:radial_schroedinger_0}, that is, $(L_{v,\ell}-k^2)G_{v,\ell}(\cdot,r')=\delta_{r'}$, see \cite{agmon1992analyticity} and \cref{sec:green_functions}.

In this article we consider equation \eqref{eq:schroedinger} with a random source function $f$. In this case the radiation solution $\psi_v$ is also a random function, as well as functions $\varphi^m_{v,\ell}$ and $f^m_\ell$ in the spherical harmonics expansions of formula \eqref{eq:harmonics_expansions}. Following \cite{gizon2017computational}, we assume that the power spectrum (power spectral density) of $\psi_v$, defined as $\mathcal P^m_{v,\ell}(r) = \mathbb E|\varphi^m_{v,\ell}(r)|^2$, can be measured experimentally. However, in contrast to \cite{gizon2017computational}, where the power spectral density is assumed to be known at the solar surface $r = R_\odot$, we assume that it can be measured at two different observation radii $R_o^\dagger > R_o \geq R_\odot$. These measurements can be roughly achieved by using concurrent MDI and HMI Dopplergrams \cite{fleck2011formation}, or multi-height measurements from raw HMI filtergrams \cite{nagashima2014interpreting}.

Our first result relates cross correlations $\mathbb E\bigl( \overline{\varphi^m_{v,\ell}(r_1)}\varphi^m_\ell(r_2) \bigr)$ to the Green's function $G_{v,\ell}(r_1,r_2)$. We prove the following proposition:

\begin{proposition}\label{thm:green_extraction} Let $v$ be a complex-valued potential satisfying \eqref{eq:potential} and let $k \in \mathbb R_+ \setminus \Sigma^P_v$ be fixed. Assume that the random functions $f^m_\ell$ satisfy the condition
\begin{equation}\label{eq:sources_equipartitioning}
  \mathbb E\bigl( \overline{f^m_\ell(r)}f^m_\ell(r+\xi)\bigr) = \Pi \delta_0(\xi) \left( -\Im \widetilde v(r) + k \delta_R(r) \right),       
\end{equation}
for some $\Pi > 0$, $R \geq R_a$. Then the following formula is valid at fixed $r_1$, $r_2>0$:
\begin{equation}\label{eq:green_extraction}
    \Pi \Im G_{v,\ell}(r_1,r_2) = \mathbb E\bigl( \overline{\varphi^m_{v,\ell}(r_1)}\varphi^m_{v,\ell}(r_2) \bigr) + O(\tfrac 1 R), \quad R \to + \infty.
\end{equation}
\end{proposition}
\Cref{thm:green_extraction} is proved in \cref{sec:proof_green_extraction}. This proposition is a variation of a well-known result, see, e.g., \cite{snieder2007extracting,gizon2017computational} and references therein. The main difference is that we consider long range potentials and the radiation Green's function, whereas in the literature the Green's function with an artificial boundary condition imposed at $r = R$ and approximating the Sommerfeld radiation condition is used. The approximate radiation boundary condition allows to get rid of the error term $O(\tfrac 1 R)$ in the formula \eqref{eq:green_extraction} but complicates the further analysis. In addition, note that in general the Sommerfeld radiation condition does not apply for long range potentials.

Assumption \eqref{eq:sources_equipartitioning} requires that the random sources be uncorrelated in space, excited throughout the volume with a power proportional to $-\Im \widetilde v$, and excited at the surface $r = R$ with a power proportional to $k$.

\begin{remark} Recall that equation \eqref{eq:schroedinger} arises, in particular, by rewriting equation \eqref{eq:acoustic} under the assumptions \eqref{eq:spaces}, \eqref{eq:atmospheric_values}, \eqref{eq:cutoff}. In this case $k$ and $v$ are given by formulas \eqref{eq:acoustic_kv}, and \Cref{thm:green_extraction} has a physical interpretation. Taking into account that $-\Im \widetilde v = 2\omega \gamma$, condition \eqref{eq:sources_equipartitioning} implies proportionality of the power spectral density of random excitations to the local attenuation (energy dissipation) rate. It has long been known in physics that this condition is related to the possibility to extract the imaginary part of the point-source response function (Green's function) from the power spectral density of the randomly excited field, which is expressed by relation \eqref{eq:green_extraction} in our setting. In physical literature similar relations are etablished in fluctuation-dissipation theorems, see, e.g., \cite{landau1996statistical}.
\end{remark}

\subsection{Uniqueness results}\label{sec:uniqueness}

\Cref{thm:green_extraction} allows to retrieve $\Im G_{v,\ell}(r,r)$ approximately from the power spectral density of noise $\mathcal P^m_{v,\ell}(r) = \mathbb E |\varphi^m_{v,\ell}(r)|^2$ at fixed $r$. Next, we prove that $\Im G_{v,\ell}(r,r)$ known exactly for all $\ell \geq 0$ and at two different $r$ uniquely determines $v$. Equivalently, taking into account the orthogonal expansion \eqref{eq:green_harmonics_expansion}, $v$ is uniquely determined by $\Im G_v$ known on $S^2_r \times S^2_r$ at two different $r$.
%
%

\begin{theorem}\label{thm:uniqueness}  let $v_1$, $v_2$ be two complex-valued potentials satisfying \eqref{eq:potential} and let $k \in \mathbb R_+ \setminus (\Sigma^P_{v_1}\cup\Sigma^P_{v_2})$ be fixed. Assume that that one of the following conditions holds true:
\begin{enumerate}
    \item[(A)] $G_{v_1} = G_{v_2}$ on $M^4_{R_o} = (S^2_{R_o} \times S^2_{R_o}) \setminus \Delta$ for some $R_o > R_a$;
    \item[(B)] $\Im G_{v_1} = \Im G_{v_2}$ on $M^4_{R_o} \cup M^4_{R_o^\dagger}$ for some $R_o^\dagger > R_o \geq  R_a$ such that $R_o^\dagger \not\in \Sigma^S_{\alpha,k,R_o}$, where $\Sigma^S_{\alpha,k,R_o} \subset [R_o,\infty)$ is a discrete set without finite accumulation points defined by \eqref{eq:coulomb_modulus_phase}, \eqref{eq:sigma_s}.
\end{enumerate}
Then $v_1 = v_2$ a.e. 
\end{theorem}

\begin{remark} If $v$ is some potential satisfying \eqref{eq:potential} then the restriction of $G_v(x,x')$ to $M^4_R$ depends on $|x-x'|$ only because $v$ is spherically symmetric. In particular, \Cref{thm:uniqueness} remains valid if the four-dimensional manifolds $M^4_R$ are replaced by the one-dimensional manifolds
\begin{equation*}
    M^1_R(x_1,x_2) = \bigl\{ (x_1 \sin \theta + x_2 \cos \theta,x_2 ) \colon \theta \in (0,\pi] \bigr\},     
\end{equation*}
for some fixed $x_1$, $x_2 \in S^2_R$ with $x_1 \cdot x_2 = 0$.
\end{remark}

\Cref{thm:uniqueness} is proved in \cref{sec:demonstration_of_uniqueness} and the proof consists of the following steps presented in \cref{sec:regular_solutions,sec:green_functions,sec:recovering_scattering_matrix,sec:recovering_dtn}. In \cref{sec:regular_solutions} we separate variables in the equation $(L_v-k^2)\psi = 0$ and establish auxilary results for the regular solutions of the arising radial Schr\"odinger equations. In \cref{sec:green_functions} we derive an appropriate relation between the Green's function $G_v$ and the Green's functions $G_{v,\ell}$ of the radial equations. This relation will allow to extract the diagonal values $G_{v,\ell}(R,R)$ from $G_v$ on $M^4_R$. Using this relation, in \cref{sec:recovering_scattering_matrix} we show that the scattering matrix elements $s_{v,\ell}$ can be extracted from $G_v$ on $M^4_{R_o}$ or from the imaginary part $\Im G_v$ only on $M^4_{R_o}\cup M^4_{R_o^\dagger}$, under the assumption that $R_o^\dagger \not\in \Sigma^S_{\alpha,k,R_o}$. In \cref{sec:recovering_dtn} we prove that the scattering matrix elements $s_{v,\ell}$ determine the Dirichlet-to-Neumann map $\Lambda_{v,R}$ for potential $v$ in some ball $B^3_R$, $R \geq R_a$, where
\begin{equation*}
    B^3_R = \bigl\{ x \in \mathbb R^3 \colon |x| < R \bigr\}.
\end{equation*}
In \cref{sec:combining_results_uniqueness} we combine these results together with the uniqueness theorem for the Dirichlet-to-Neumann map from \cite{novikov1988multidimensional} to uniquely determine $v$. This will prove \Cref{thm:uniqueness}.

\begin{corollary}\label{thm:uniqueness_acoustic} Let $c$, $\rho$, $\gamma$, and $c'$, $\rho'$, $\gamma'$ be two sets of parameters satisfying \eqref{eq:spaces}, \eqref{eq:atmospheric_values} and define the corresponding potentials $v=v_\omega$, $v' = v_\omega'$ and the wavenumber $k = k_\omega$ according to formula \eqref{eq:acoustic_kv} at fixed $\omega$. Let $\omega_1 \neq \omega_2$ be two positive frequencies satisfying \eqref{eq:cutoff} and such that $k_\omega \in \mathbb R_+ \setminus (\Sigma^P_{v_\omega} \cup \Sigma^P_{v_\omega'})$ for $\omega = \omega_1$, $\omega_2$. Let $G_{v_\omega}$, $G_{v_\omega'}$ be the radiation Green's functions at fixed $\omega$ for the potentials $v$, $v'$ respectively. 
    
Suppose that $\Im G_{v_\omega} = \Im G_{v_\omega'}$ on $M^4_{R_0} \cup M^4_{R_o^\dag}$ for some $R_o^\dag > R_o \geq R_a$ such that $R_o^\dag \not\in\Sigma^S_{1/H,k_\omega,R_o}$, where $\omega = \omega_1$, $\omega_2$. Then $c=c'$, $\rho = \rho'$, $\gamma = \gamma'$ a.e.
\end{corollary}

\begin{proof}
Under the assumptions of \Cref{thm:uniqueness_acoustic} it follows from \Cref{thm:uniqueness} that $v_\omega = v_\omega'$ for $\omega = \omega_1$, $\omega_2$. Using that $\Re v_\omega = \Re v_\omega'$ for $\omega = \omega_1$, $\omega_2$ one can show that $c = c'$, $\rho = \rho'$, see, e.g., the proof of \cite[Theorem 2.9]{agaltsov2018monochromatic}. Then, recalling that $\Im v_\omega = -2\omega\gamma/c^2$, $\Im v_\omega' = -2\omega\gamma'/(c')^2$, it follows from the equality $\Im v_\omega = \Im v_\omega'$ for $\omega = \omega_1$, $\omega_2$ together with the equality $c = c'$ that $\gamma = \gamma'$, concluding the proof of \Cref{thm:uniqueness_acoustic}. 
\end{proof}

In \cref{sec:reconstruction} we shall present numerical simulations which confirm the uniqueness results of \Cref{thm:uniqueness} and \Cref{thm:uniqueness_acoustic}.

\section{Proof of the main results}\label{sec:demonstration_of_uniqueness}
\subsection{Properties of regular radial solutions}\label{sec:regular_solutions}

In this subsection we shall establish some auxilary results regarding regular solutions of the radial Schr\"odinger equation which arises by separation of variables in the homogeneous equation $(L_v - k^2)\psi = 0$.

As the potential $v$ is spherically symmetric, this equation separates in spherical coordinates. We seek a solution $\psi^m_{v,\ell}\in H^2_\text{loc}(\mathbb R^3)$ of the form
\begin{equation}\label{eq:angular_solution}
   \psi^m_{v,\ell}(x) = \tfrac{1}{|x|} \varphi_{v,\ell}(|x|) Y^m_\ell(\tfrac{x}{|x|}),
\end{equation}
which leads to the equation
\begin{equation}\label{eq:radial_schroedinger}
    \bigl(L_{v,\ell} - k^2\bigr) \varphi_{v,\ell} = 0,
\end{equation}
together with the condition that $\varphi_{v,\ell}$ vanishes at the origin. One can show that this determines $\varphi_{v,\ell} \in H^2_\text{loc}(\mathbb R_+)$ uniquely up to a multiplicative factor, see, e.g., \cite{agmon1992analyticity}. We impose the boundary condition
\begin{equation}\label{eq:regular_asymptotics}
    \lim_{r\to+0} r^{-\ell-1}\varphi_{v,l}(r) = 1,
\end{equation}
which fixes $\varphi_{v,\ell}$ uniquely.

 Note that $\varphi_{v,\ell}(r)$ does not depend on the values of $\widetilde v$ in the region $[r,+\infty)$, see \cite[formula (12.4)]{newton1982scattering}. This analysis implies the following lemma.



\begin{lemma}\label{thm:regular_solution} Let $k > 0$ and let $v$ be a complex-valued potential satisfying \eqref{eq:potential}. Then the Dirichlet problem 
\begin{equation*}
    (L_v - k^2)\psi = 0 \; \text{in $B^3_R$}, \quad \psi|_{S^2_R} = Y^m_\ell,
\end{equation*}
where $Y^m_\ell=Y^m_\ell(x/|x|)$, $x \in S^2_R$, has a unique solution $\psi \in H^2(B^3_R)$ if and only if $\varphi_{v,\lambda}(R) \neq 0$ for all integer $\lambda \geq 0$. In addition, this solution is given by the formula
\begin{equation}\label{eq:regular_solution}
    \psi(x) = \frac{R}{|x|}\frac{\varphi_{v,\ell}(|x|)}{\varphi_{v,\ell}(R)} Y^m_\ell(\tfrac{x}{|x|}).
\end{equation}
\end{lemma}

Next we shall derive an expression for the regular solution $\varphi_{v,\ell}$ in the domain $r\geq R$ in terms of the Coulomb wave functions and of the so-called scattering matrix element $s_{v,\ell}$. First we recall the definition and some basic properties of the Coulomb wave functions from \cite{erdelyi1957asymptotic,olver2010nist}.

The Coulomb wave functions $H^\pm_\ell(\eta,kr)$, $\eta = \alpha/(2k)$, are the unique solutions of equation \eqref{eq:radial_schroedinger} with $\widetilde v(r)=\alpha/r$ specified by the following asymptotics as $r \to + \infty$:
\begin{subequations}
\begin{gather}\label{eq:coulomb_asymptotics}
    H^\pm_\ell(\eta,kr) = \exp(\pm i \theta_\ell(\eta,kr)) + O(\tfrac{1}{r}),\\
    \theta_\ell(\eta,kr) = kr - \eta\ln(2kr) - \tfrac 1 2 \ell \pi + \sigma_\ell(\eta),\label{eq:coulomb_theta}
\end{gather}
\end{subequations}
where $\sigma_\ell(\eta) = \arg \Gamma(\ell+1+i\eta)$ is the Coulomb phase shift and $\Gamma$ denotes the usual gamma function. Functions $H^+_\ell(\eta,kr)$ and $H^-_\ell(\eta,kr)$ are complex conjugates of each other and are linearly independent. Using \eqref{eq:coulomb_asymptotics} together with the relation 
\begin{equation*}
    \frac{\partial}{\partial \rho} H^\pm_\ell(\eta,\rho) = \left(\frac{\ell+1}{\rho}+\frac{\eta}{\ell+1}\right) H^\pm_\ell(\eta,\rho) - \sqrt{1+\frac{\eta^2}{(\ell+1)^2}} H^\pm_{\ell+1}(\eta,\rho), \quad \ell \geq 0
\end{equation*}
given in \cite{powell1947recurrence}, one can show that
\begin{equation}\label{eq:coulomb_radiation}
    \frac{\partial}{\partial r} H^\pm_\ell(\eta,kr) = \pm ik H^\pm_\ell(\eta,kr) + O(\tfrac{1}{r}), \quad r \to+\infty.
\end{equation}
Using these properties, we shall prove the following result.

\begin{lemma}\label{thm:regular_solution_expansion} Let $v$ be a complex-valued potential satisfying \eqref{eq:potential}, let $k \in \mathbb R_+\setminus \Sigma^P_v$ be fixed and let $\eta = \alpha/(2k)$. Then the function $\varphi_{v,\ell}$ defined by \eqref{eq:radial_schroedinger}, \eqref{eq:regular_asymptotics} admits in the region $r \geq R_a$ the representation
    \begin{equation}\label{eq:regular_solution_expansion}
      \varphi_{v,\ell}(r) = b_\ell\left( H^-_\ell(\eta,kr) - s_{v,\ell} H^+_\ell(\eta,kr) \right), \quad r \geq R_a,
    \end{equation}
with unique $b_\ell = b_\ell(k) \in \mathbb C \setminus \{0\}$ and $s_{v,\ell}=s_{v,\ell}(k) \in \mathbb C \setminus \{0\}$.
\end{lemma} 

\begin{proof} As the Coulomb wave functions $H^\pm_\ell(\eta,kr)$ are linearly independent, any solution to equation \eqref{eq:radial_schroedinger} in the region $r \geq R_a$ is given by their linear combination with unique coefficients. In particular, $\varphi_{v,\ell}$ can be expressed in the region $r \geq R$ in the form
\begin{equation*}
    \varphi_{v,\ell}(r) = a_{v,\ell} H^+_\ell(\eta,kr) + b_{v,\ell} H^-_\ell(\eta,kr), \quad r \geq R,
\end{equation*}
for some $a_{v,\ell}$, $b_{v,\ell} \in \mathbb C$. We shall show that $a_{v,\ell} \neq 0$, $b_{v,\ell} \neq 0$.

Recall from \cite{saito1974principle} that for any $k \in \mathbb R_+$ outside the singular set $\Sigma^P_v$ and for any $f \in L^2_{1+\varepsilon}(\mathbb R^3)$ the Schr\"odinger equation \eqref{eq:schroedinger} admits the unique solution $\psi \in H^2_\text{loc}(\mathbb R^3) \cap L^2_{-1-\varepsilon}(\mathbb R^3)$, $\varepsilon \in (0,\tfrac 1 2]$, satisfying the radiation condition
\begin{subequations}
\begin{gather}\label{eq:radiation}
  \int_{|x|\geq 1}(1+|x|)^{-1+\varepsilon}|\mathcal D\psi(x)|^2 \, dx < \infty, \quad \varepsilon \in (0,\tfrac 1 2], \\
     \mathcal D \psi(x) = \nabla \psi(x) + \tfrac{x}{|x|^2} \psi(x) - i k \tfrac{x}{|x|}\psi(x).
\end{gather}
\end{subequations}
Now assume that $b_{v,\ell} = 0$. Then it follows from \eqref{eq:coulomb_asymptotics}, \eqref{eq:coulomb_radiation} that the function defined by \eqref{eq:regular_solution} is a non-zero solution of class $H^2_\text{loc}(\mathbb R^3) \cap L^2_{-1-\varepsilon}(\mathbb R^3)$, $\varepsilon \in (0,\tfrac 1 2]$, to $(L_v-k^2)\psi = 0$ in $\mathbb R^3$ satisfying the radiation condition \eqref{eq:radiation}. This contradicts the assumption $k \not\in \Sigma^P_v$.
    
Now assume that $a_{v,\ell} = 0$ and let $\psi$ be defined by \eqref{eq:regular_solution}. Then it follows from \eqref{eq:coulomb_asymptotics}, \eqref{eq:coulomb_radiation} that $\overline \psi$ is of class $H^2_\text{loc}(\mathbb R^3) \cap L^2_{-1-\varepsilon}(\mathbb R^3)$, $\varepsilon \in (0,\tfrac 1 2]$, and satisfies $(L_{\overline v}-k^2)\overline \psi = 0$ in $\mathbb R^3$ together with the radiation condition \eqref{eq:radiation}. Taking into account definition \eqref{eq:Sigma_definition}, this also contradicts the assumption $k \not\in \Sigma^P_v$ and concludes the proof of \Cref{thm:regular_solution_expansion}.
\end{proof}

\begin{remark}\label{rmk:scattering_matrix}
The coefficient $s_{v,\ell} = s_{v,\ell}(k)$ in \Cref{thm:regular_solution_expansion} is called the $\ell$-th scattering matrix element of the potential $v$.
\end{remark}

\subsection{Green's functions}\label{sec:green_functions}

In this subsection we shall express the radiation Green's function $G_{v,\ell}$ for equation \eqref{eq:radial_schroedinger} in the region $r \geq R_a$ in terms of the Coulomb wave functions. Then we shall give a formula for extracting the diagonal values $G_{v,\ell}(R,R)$ from the radiation Green's functon $G_v$ for equation \eqref{eq:schroedinger} known at $M^4_R$.


In addition to the regular solution $\varphi_{v,\ell}$ of equation \eqref{eq:schroedinger} specified by the boundary condition \eqref{eq:regular_asymptotics}, we consider the outgoing solution $\varphi^+_{v,\ell}$ which is specified by the asymptotics
\begin{equation}\label{eq:outgoing_asymptotics}
    \varphi^+_{v,\ell}(r)=H^+_\ell(\eta,kr), \quad r \geq R_a.
\end{equation}
The outgoing Green's function $G_{v,\ell}(r,r')$ for equation \eqref{eq:radial_schroedinger} is defined as a distributional solution to the equation $(L_{v,\ell}-k^2)G_{v,\ell}(\cdot,r')=\delta_{r'}$ specified by the following boundary conditions at fixed $r' > 0$:
\begin{equation*}
    G_{v,\ell}(r,r') = O(r^{\ell+1}), \; r \to+0, \quad G_{v,\ell}(r,r') = c H^+_\ell(\eta,kr), \; r \to + \infty,
\end{equation*}
for some non-zero constant $c = c(r',\eta,k)$. If the regular solution $\varphi_{v,\ell}(r)$ and the outgoing solution $\varphi^+_{v,\ell}(r)$ are linearly independent, this Green's function exists, is unique and is given by the explicit formula
\begin{equation}\label{eq:radial_green_definition}
    G_{v,\ell}(r,r') = -\frac{\varphi_{v,\ell}(r_<) \varphi^+_{v,\ell}(r_>)}{[\varphi_{v,\ell},\varphi^+_{v,\ell}]}, \quad r, r' > 0,
\end{equation}
where $r_< = \min(r,r')$, $r_> = \max(r,r')$ and the Wronskian
\begin{equation}\label{eq:wronskian}
    [\varphi_{v,\ell},\varphi^+_{v,\ell}] = \varphi_{v,\ell}(r)\frac{\partial \varphi^+_{v,\ell}(r)}{\partial r} - \frac{\partial \varphi_{v,\ell}(r)}{\partial r}\varphi^+_{v,\ell}(r)
\end{equation}
is independent of $r$. Note that formula \eqref{eq:radial_green_definition} is a standard result from the theory of Sturm-Liouville problems, see, e.g., \cite[p. 158, formula (5.65)]{teschl2012ordinary}.


\begin{lemma}\label{thm:radial_green} If $k \in \mathbb R_+\setminus\Sigma^P_v$, then $G_{v,\ell}$ is well-defined and is given in the region $r \geq R_a$, $r' \geq R_a$ by the formula
    \begin{equation}\label{eq:radial_green}
        G_{v,\ell}(r,r') = \frac{i}{2k}\left( H^-_\ell(\eta,kr_<) - s_{v,\ell} H^+_\ell(\eta,kr_<) \right) H^+_\ell(\eta,kr_>), \quad r, r' \geq R_a.
    \end{equation}
\end{lemma}
\begin{proof} It follows from \Cref{thm:regular_solution_expansion} that under the assumption $k \in \mathbb R_+\setminus \Sigma^P_v$ the functions $\varphi_{v,\ell}(r)$ and $H^+_\ell(\eta,kr)$ are linearly independent in the region $r \geq R_a$. Using the relation $[H^-_\ell(\eta,kr),H^+_\ell(\eta,kr)]=2ik$, given in \cite{olver2010nist}, and formula \eqref{eq:regular_solution_expansion} we get $[\varphi_{v,\ell},\varphi^+_{v,\ell}] = 2ikb_{v,\ell}$. Together with \eqref{eq:radial_green_definition}, this implies \eqref{eq:radial_green}.
\end{proof}

Next we shall show how the diagonal values $G_{v,\ell}(R,R)$ can be extracted from the Green's function $G_v$ restricted to $M^4_R$.

First recall that the Legendre polynomials $P_\ell$ can be defined using the formal generating identity \cite{szego1939orthogonal}
\begin{equation}\label{eq:legendre}
  \frac{1}{\sqrt{1- 2st + t^2}} = \sum_{\ell=0}^\infty P_\ell(s) t^\ell. 
\end{equation}
Using the Laplace formula \cite[Theorem 8.21.2]{szego1939orthogonal} and the Dirichlet convergence test one can show that at fixed $t=1$ this series converges pointwise for all $s \in (-1,1)$. Besides, the Legendre polynomials form a complete orthogonal system in $L^2(-1,1)$ such that
\begin{equation}\label{eq:legendre_normalization}
    \int_{-1}^1 P_\ell(s)P_m(s) \, ds = \frac{2}{2\ell+1} \delta_{\ell m},
\end{equation}
where $\delta_{\ell m}$ is the Kronecker delta.

Now we recall \cite{agmon1992analyticity} that at fixed $k \in \mathbb R_+\setminus\Sigma^P_v$ the Green's function $G_v(x,x')$ is continuous outside of the diagonal $x=x'$ and $G_v(x,x')=O(|x-x'|^{-1})$ as $x\to x'$. Besides, the series expansion
\begin{equation}\label{eq:multipole_expansion}
    G_v(x, x') = \frac{1}{4\pi R^2} \sum_{l=0}^\infty (2\ell+1) G_{v,\ell}(R,R)P_\ell(x \cdot x' / R^2), \quad (x,x')\in M^4_R,
\end{equation}
converges for $x \neq \pm x'$ and $G_{v,\ell}(R,R)$ has the asymptotics
\begin{equation}\label{eq:partial_green_asymptotics}
    G_{v,\ell}(R,R) = \frac{R}{2\ell+1}\bigl(1 + O(\tfrac 1 \ell ) \bigr), \quad l \to + \infty.
\end{equation}
Note that series expansion \eqref{eq:multipole_expansion} follows from the spherical harmonics expansion \eqref{eq:green_harmonics_expansion}, taking into account the well known addition theorem (see, e.g., \cite{agmon1992analyticity}): 
\begin{equation*}
    P_\ell(\vartheta \cdot \vartheta') = \frac{4\pi}{2\ell+1} \sum_{m=-\ell}^\ell Y_{\ell}^m(\vartheta) \overline{Y^m_\ell(\vartheta')}, \quad \vartheta, \vartheta' \in S^2_1.
\end{equation*}

\begin{lemma}\label{thm:green_separation} Let $k \in \mathbb R_+ \setminus \Sigma^P_v$ and fix $x'$, $x'' \in S^2_R$ such that $x' \cdot x'' = 0$. Then for each $\ell \geq 0$ 
\begin{gather*}
    G_{v,\ell}(R,R) = \frac{R}{2\ell+1} -  2\pi R^2 \int_0^\pi g_{v,R}(\cos\theta)P_\ell(\cos\theta) \sin\theta d\theta, \\
    g_{v,R}(\cos \theta)=G_v(x'\cos \theta + x''\sin \theta,x') - \frac{1}{4\pi R}\frac{1}{\sqrt{2-2\cos\theta}}, \; \theta\in(0,\pi).
\end{gather*}
\end{lemma}
\begin{proof} Using \eqref{eq:legendre} and \eqref{eq:multipole_expansion} we get a pointwise convergent series expansion
\begin{equation*}
    g_{v,R}(s) = \frac{1}{4\pi R^2}\sum_{l=0}^\infty (2\ell+1)\bigl(G_{v,\ell}(R,R)-\tfrac{R}{2\ell+1} \bigr) P_\ell(s), \quad s \in (-1,1),
\end{equation*}
which, in view of \eqref{eq:legendre_normalization} and \eqref{eq:partial_green_asymptotics}, also converges in $L^2(-1,1)$. Recalling that Legendre polynomials form a complete orthogonal system in $L^2(-1,1)$, we get \Cref{thm:green_separation}.
\end{proof}

\subsection{Extracting the Green's function from cross correlations}\label{sec:proof_green_extraction}
In this subsection we prove \Cref{thm:green_extraction}. First, recall that the Green's function $G_{v,\ell}$ satisfies the reciprocity relation
\begin{equation}\label{eq:green_reciprocity}
    G_{v,\ell}(r_1,r_2) = G_{v,\ell}(r_2,r_1), \quad r_1, r_2 > 0.
\end{equation}
To prove it, consider the equations
\begin{equation*}
    \bigl(L_{v,\ell} - k^2 \bigr) G_{v,\ell}(\cdot,r_1) = \delta_{r_1}, \quad \bigl(L_{v,\ell} - k^2 \bigr) G_{v,\ell}(\cdot,r_2) = \delta_{r_2}.
\end{equation*}
Multiplying the first equation by $G_{v,\ell}(\cdot,r_2)$, subtracting the second equation multiplied by $G_{v,\ell}(\cdot,r_1)$, and integrating over $(0,R)$, $R > r_1$, $r_2$, we obtain
\begin{equation*}
   \lbrack G_{v,\ell}(\cdot,r_1),G_{v,\ell}(\cdot,r_2)\rbrack \bigr|^R_{+0}  = G_{v,\ell}(r_1,r_2) - G_{v,\ell}(r_2,r_1),
\end{equation*}
where $[-,-]$ denotes the Wronskian defined according to \eqref{eq:wronskian} and the notation $x|^b_a = x(b)-x(a)$ is used. The next step is to show that the term on the left hand side vanishes. Using formulas \eqref{eq:regular_asymptotics}, \eqref{eq:radial_green_definition} and the estimate $\tfrac{\partial}{\partial r}\varphi_{v,\ell}(r) = O(r^\ell)$, $r\to+0$, given in \cite[Theorem 3.3]{agmon1992analyticity}, we get $\lbrack G_{v,\ell}(\cdot,r_1),G_{v,\ell}(\cdot,r_2)\rbrack(+0)=0$. Using formulas \eqref{eq:coulomb_asymptotics}, \eqref{eq:coulomb_theta}, \eqref{eq:coulomb_radiation}, \eqref{eq:radial_green_definition}, we also get $\lbrack G_{v,\ell}(\cdot,r_1),G_{v,\ell}(\cdot,r_2)\rbrack(R)=O(\tfrac 1 R)$, $R\to+\infty$. As $R$ tends to $+\infty$, we get formula \eqref{eq:green_reciprocity}.

To prove \eqref{eq:green_extraction}, we follow a similar scheme. We start from the equations
\begin{equation*}
    \bigl(\overline{L_{v,\ell}} - k^2 \bigr) \overline{G_{v,\ell}(\cdot,r_1)} = \delta_{r_1}, \quad \bigl(L_{v,\ell} - k^2 \bigr) G_{v,\ell}(\cdot,r_2) = \delta_{r_2}.
\end{equation*}
Multiplying the first equation by $G_{v,\ell}(\cdot,r_2)$, subtracting the second equation multiplied by $\overline{G_{v,\ell}(\cdot,r_1)}$, integrating over $(0,R)$, $R > r_1$, $r_2$, we get
\begin{equation}\label{eq:green_extraction_1}
    \begin{gathered}
   \lbrack\overline{G_{v,\ell}(\cdot,r_1)},G_{v,\ell}(\cdot,r_2)\rbrack\bigr|^R_{+0} - 2i \int_0^R \Im \widetilde v(r)\overline{G_{v,\ell}(r,r_1)}G_{v,\ell}(r,r_2) dr \\   = G_{v,\ell}(r_1,r_2) - \overline{G_{v,\ell}(r_2,r_1)}.
   \end{gathered}
\end{equation}
In a similar way with the proof of formula \eqref{eq:green_reciprocity}, one can show that the Wronskian vanishes at zero and that
\begin{equation*}
    [\overline{G_{v,\ell}(\cdot,r_1)},G_{v,\ell}(\cdot,r_2)](R) = 2ik \overline{G_{v,\ell}(R,r_1)}G_{v,\ell}(R,r_2) + O(\tfrac 1 R), \quad R \to + \infty.
\end{equation*}
Combining this with formulas \eqref{eq:radial_radiation_solution}, \eqref{eq:sources_equipartitioning}, \eqref{eq:green_extraction_1}, \eqref{eq:green_reciprocity} to compute $\mathbb E\bigl( \overline{\psi^m_\ell(r_1)}\psi^m_\ell(r_2) \bigr)$, we get formula \eqref{eq:green_extraction}, which concludes the proof of \Cref{thm:green_extraction}.

\subsection{Recovering the scattering matrix elements}\label{sec:recovering_scattering_matrix}
In this subsection we shall show that the scattering matrix elements $s_{v,\ell}$ for the potential $v$ can be extracted from the Green's function $G_v$ on $M^4_{R_o}$ or from its imaginary part $\Im G_v$ only on $M^4_{R_o} \cup M^4_{R_o^\dagger}$, where $R_a \leq R_o < R_o^\dagger$.

Recall that the Coulomb function $H^+_\ell(\eta,kr)$ does not vanish for $r > 0$, since $H^+_\ell(\eta,kr)$ and its complex conjugate $H^-_\ell(\eta,kr)$ form a basis of solutions of equation \eqref{eq:radial_schroedinger} with $\widetilde v(r) = \alpha/r$. Together with Lemmas \ref{thm:radial_green} and \ref{thm:green_separation}, this leads to the following result.
\begin{lemma}\label{thm:scattering_matrix_extraction_a} Let $v_1$, $v_2$ be two complex-valued potentials satisfying \eqref{eq:potential}, let $k \in \mathbb R_+\setminus(\Sigma^P_{v_1} \cup \Sigma^P_{v_2})$ be fixed and let $R_o \geq R_a$. Suppose that $G_{v_1} = G_{v_2}$ on $M^4_{R_o}$. Then $G_{v_1,\ell}(R_o,R_o)=G_{v_2,\ell}(R_o,R_o)$ and $s_{v_1,\ell}=s_{v_2,\ell}$ for all $\ell \geq 0$. 
\end{lemma}

Next we shall show that the scattering matrix elements $s_{v,\ell}$ can be extracted from $\Im G_v$ only, i.e., without knowing $\Re G_v$. However, the values of $\Im G_v$ must be given not only on $M^4_{R_o}$ but also on $M^4_{R_o^\dagger}$ for some $R_o^\dagger > R_o$.

Note that formula \eqref{eq:radial_green} implies that
\begin{subequations}
\begin{gather}\label{eq:scattering_matrix_equation}
    \cos\vartheta_\ell(\eta,kr) \Re s_{v,\ell} - \sin \vartheta_\ell(\eta,kr) \Im s_{v,\ell} 
     = \frac{H^-_\ell(\eta,kr)-2k\Im G_{v,\ell}(r,r)}{|H^+_\ell(\eta,kr)|^2}, \\
     H^+_\ell(\eta,kr) = |H^+_\ell(\eta,kr)|\exp(i\vartheta_\ell(\eta,kr)/2\bigr),\label{eq:coulomb_modulus_phase}
\end{gather}
\end{subequations}
where $r \geq R_a$, $\eta = \alpha/(2k)$. Using \eqref{eq:scattering_matrix_equation} with $r = R_o$, $R_o^\dagger$ one can see that $s_{v,l}$ is uniquely determined from $\Im G_{v,\ell}(R_o,R_o)$ and $\Im G_{v,\ell}(R_o^\dagger,R_o^\dagger)$ if and only if $\sin(\vartheta_\ell(\eta,kR_o^\dagger)-\vartheta_\ell(\eta,kR_o)) \neq 0$. This justifies the definition of the singular set
\begin{equation}\label{eq:sigma_s}
    \Sigma^S_{\alpha,k,R_o} = \bigl\{ r \in [R_o,+\infty) \colon \sin(\vartheta_\ell(\eta,kr)-\vartheta_\ell(\eta,kR_o))=0 \; \text{for some $\ell \geq 0$} \bigr\}.
\end{equation}

\begin{lemma}\label{thm:scattering_matrix_extraction_b} Let $v_1$, $v_2$ be two complex-valued potentials satisfying \eqref{eq:potential}, let $k \in \mathbb R_+ \setminus (\Sigma^P_{v_1}\cup\Sigma^P_{v_2})$, and let $R_o^\dagger > R_o \geq R_a$ be such that $R_o^\dagger \not\in \Sigma^S_{\alpha,k,R_o}$. Suppose that $\Im G_{v_1}=\Im G_{v_2}$ on $M^4_{R_o} \cup M^4_{R_o^\dagger}$. Then $s_{v_1,\ell}=s_{v_2,\ell}$ for all $\ell \geq 0$. Besides, the set $\Sigma^S_{\alpha,k,R_o}$ is discrete and does not have finite accumulation points.
\end{lemma}
\begin{proof} Under the assumptions of \Cref{thm:scattering_matrix_extraction_b} it follows from \Cref{thm:green_separation} that for all $\ell \geq 0$ the equality $\Im G_{v_1,\ell}(r,r) = \Im G_{v_2,\ell}(r,r)$ holds true for $r = R_o$, $R_o^\dagger$. Together with the discussion before \Cref{thm:scattering_matrix_extraction_b} it implies that $s_{v_1,\ell}=s_{v_2,\ell}$ for all $\ell \geq 0$. This concludes the proof of the first assertion of \Cref{thm:scattering_matrix_extraction_b}.
    
Next we shall prove the second assertion. Using formulas (33.2.11), (33.5.8), (33.5.9) of \cite{olver2010nist} one can see that
\begin{equation}\label{eq:determinant_asymptotics}
    \sin\bigl(\vartheta_\ell(\eta,kr)-\vartheta_\ell(\eta,kR_o)\bigr) \sim e^{-1-\pi\eta} \left(\frac{ekr}{2\ell}\right)^{2\ell+1}, \quad \ell \to + \infty,
\end{equation}
locally uniformly in $r\in\mathbb R_+$. Besides, it follows from formula (33.2.11) of \cite{olver2010nist} and from discussion below it that at fixed $\ell \geq 0$ the set of solutions $r \in \mathbb R_+$ of the equation $\sin(\vartheta_\ell(\eta,kr)-\vartheta_\ell(\eta,kR_o))=0$ is discrete and does not have finite accumulation points, as the zero-set of a non-zero analytic function. Together with \eqref{eq:determinant_asymptotics}, this concludes the proof of \Cref{thm:scattering_matrix_extraction_b}.
\end{proof}

\subsection{Recovering the Dirichlet-to-Neumann map}\label{sec:recovering_dtn}
In this subsection we shall show that the scattering matrix elements $s_{v,\ell}$ uniquely determine the Dirichlet-to-Neumann map for $v$ is some ball $B^3_R$, $R \geq R_a$.

Note that \Cref{thm:regular_solution} justifies the definition of the singular set
\begin{equation*}
    \Sigma^D_{v,k} = \bigl\{ R > 0 \colon \varphi_{v,\ell}(R) = 0 \; \text{for some $\ell \geq 0$} \bigr\}.
\end{equation*}
\begin{lemma}\label{thm:dirichlet_uniqueness} Let $v$ be a complex-valued potential satisfying \eqref{eq:potential} and let $k \in \mathbb R_+$ be fixed. Then $R \in \mathbb R_+\setminus \Sigma^D_{v,k}$ if and only if for any $g \in H^{3/2}(S^2_R)$ the Dirichlet problem 
\begin{equation}\label{eq:dirichlet_problem}
    (L_v - k^2) \psi = 0 \; \text{in $B^3_R$}, \quad \psi|_{S^2_R} = g,
\end{equation}
is uniquely solvable for $\psi \in H^2(B^3_R)$. Besides, the set $\Sigma^D_{v,k}$ is discrete and does not have finite accumulation points.
\end{lemma}
\begin{proof} It follows from \Cref{thm:regular_solution} that if the Dirichlet problem \eqref{eq:dirichlet_problem} is uniquely solvable for any $g \in H^{3/2}(S^2_R)$ then $R \not\in \Sigma^D_{v,k}$.
    
Now suppose that $R \not\in \Sigma^D_{v,k}$. First we shall show that \eqref{eq:dirichlet_problem} does not admit a non-zero solution for $g = 0$. Assuming that $\psi \in H^2_0(B^3_R)$ is a solution to \eqref{eq:dirichlet_problem} with $g=0$ one can see (taking into account the spherical symmetry of $v$) that its partial wave components
\begin{equation*}
    \psi^m_\ell(x) = Y^m_\ell(x/|x|) \int_{S^2_1} \psi(|x|\omega) \overline Y{}^m_\ell(\omega) \, dS_\omega
\end{equation*}
belong to $H^1_0(B^3_R)$ and satisfy \eqref{eq:dirichlet_problem} weakly. Because of the boundary elliptic regularity \cite{evans2010partial} they also belong to $H^2(B^3_R)$. Then it follows from \Cref{thm:regular_solution} that all $\psi^m_\ell$ vanish and $\psi = 0$.

Next we recall that the operator 
\begin{equation}\label{eq:pde_operator}
    \psi \mapsto \bigl( (L_v-k^2) \psi, \psi|_{S^2_R} \bigr)
\end{equation}
is Fredholm of index zero from $H^2(B^3_R)$ to $L^2(B^3_R) \times H^{3/2}(S^2_R)$, see \cite{grisvard1985elliptic}. Together with already established uniqueness for the Dirichlet problem \eqref{eq:dirichlet_problem} for $R \not\in \Sigma^D_{v,k}$, this proves the first assertion of \Cref{thm:dirichlet_uniqueness}.


Next we shall show that $\Sigma^D_{v,k}$ is a discrete set without finite accumulation points. It can be shown \cite{agmon1992analyticity} that the regular solution $\varphi_{v,\ell}(r)$ defined by \eqref{eq:radial_schroedinger}, \eqref{eq:regular_asymptotics} satisfies the estimate
\begin{equation*}
    \varphi_{v,\ell}(r) = r^{\ell+1} ( 1 + O(\tfrac 1 \ell)), \quad \ell \to + \infty,
\end{equation*}
uniformly in $r \in (0,R]$ at fixed $R>0$. Besides, zeros of each $\varphi_{v,\ell}$ are discrete and do not have finite accumulation points. This concludes the proof of the second assertion of \Cref{thm:dirichlet_uniqueness}.
\end{proof}

\begin{remark}\label{rmk:solution_estimate} Recalling from \cite{grisvard1985elliptic} that the operator \eqref{eq:pde_operator} is Fredholm of index zero from $H^2(B^3_R)$ to $L^2(B^3_R) \times H^{3/2}(S^3_R)$ one can see that if the potential $v \in L^\infty(B^3_R)$ is such that the Dirichlet problem \eqref{eq:dirichlet_problem} with $g \in H^{3/2}(S^2_R)$ is uniquely solvable for $\psi \in H^2(B^3_R)$, then 
\begin{equation*}
        \|\psi\|_{H^2(B^3_R)} \leq C_{v,k,R} \|g\|_{H^{3/2}(S^2_R)}
\end{equation*}
for some constant $C_{v,k,R}>0$. In addition, the trace theorem \cite{grisvard1985elliptic} leads to the estimate
\begin{equation*}
    \|\tfrac{\partial \psi}{\partial r}\|_{H^{1/2}(S^2_R)} \leq C'_{v,k,R}\|g\|_{H^{3/2}(S^2_R)}
\end{equation*}
for some constant $C'_{v,k,R}>0$, where $\tfrac{\partial\psi}{\partial r} = \tfrac{x}{|x|}\nabla \psi$ is the derivative of $\psi$ in the radial direction.
\end{remark}

Under the assumption that the Dirichlet problem \eqref{eq:dirichlet_problem} is uniquely solvable for $\psi \in H^2(B^3_R)$ for all $g \in H^{3/2}(S^2_R)$, we define the Dirichlet-to-Neumann map $\Lambda_{v,R} \in \mathcal L\bigl(H^{3/2}(S^2_R),H^{1/2}(S^2_R)\bigr)$ by $\Lambda_{v,R} \varphi = \tfrac{\partial \psi}{\partial r}|_{S^2_R}$. Next we shall show that the partial scattering matrix elements $s_{v,\ell}$ known for all $\ell \geq 0$ uniquely determine the Dirichlet-to-Neumann map $\Lambda_{v,R}$.

\begin{lemma}\label{thm:dtn_extraction} Let $v_1$, $v_2$ be two complex-valued potentials satisfying \eqref{eq:potential} and let $k \in \mathbb R_+ \setminus (\Sigma^P_{v_1} \cup \Sigma^P_{v_2})$ be fixed. Besides, let $R \in [R_a,+\infty) \setminus (\Sigma^D_{v_1,k}\cup\Sigma^D_{v_2,k})$. Suppose that $s_{v_1,\ell}=s_{v_2,\ell}$ for all $\ell\geq 0$. Then $\Lambda_{v_1,R}=\Lambda_{v_2,R}$.
\end{lemma}
\begin{proof} It follows from \Cref{thm:regular_solution,thm:regular_solution_expansion} and from continuity of $\varphi_{v,\ell}(r)$ at $r=R$ that $\Lambda_{v_1,R}|_{\mathcal H_{\ell,R}} = \Lambda_{v_2,R}|_{\mathcal H_{\ell,R}}$, where $\mathcal H_{\ell,R}$ denotes the space of restrictions to $S^2_R$ of harmonic polynomials of degree $\ell$, spanned by the spherical harmonics $Y^m_\ell = Y^m_\ell(\tfrac{x}{|x|})$, $|m|\leq \ell$. More precisely, the following explicit formula for $\Lambda_{v_j,R}|_{\mathcal H_{\ell,R}}$ is valid:
\begin{equation}\label{eq:dtn_on_harmonics}
        \Lambda_{v_j,R} Y^m_\ell =  R\frac{\bigr[\tfrac{\partial}{\partial r}(\tfrac 1 r H^-_\ell(\eta,kr)) - s_{v_j,\ell} \tfrac{\partial}{\partial r}(\tfrac 1 r H^+_\ell(\eta,kr)) \bigr]_{r=R}}{H^-_\ell(\eta,kR) - s_{v_j,\ell} H^+_\ell(\eta,kR)} Y^m_\ell,
\end{equation}
where $\eta = \alpha/(2k)$ and the denominator is non-zero as $R \not\in \Sigma^D_{v_j,k}$.
        
Now let $g \in H^{3/2}(S^2_R)$ and denote by $\psi_{v_j} \in H^2(B^3_R)$ the unique solution of the Dirichlet problem \eqref{eq:dirichlet_problem} with $v = v_j$. Besides, define $g_N$ by
\begin{equation*}
     g_N = \sum_{\ell=0}^N \sum_{|m|\leq\ell}g^m_\ell Y^m_\ell, \quad g^m_\ell = \int_{S^2_1}g(R\omega)\overline Y{}^m_\ell(\omega)dS_\omega.
\end{equation*}
so that $g_N \to g$ in $H^{3/2}(S^2_R)$ and, according to \Cref{rmk:solution_estimate}, $\Lambda_{v_j,R}g_N \to \Lambda_{v_j,R}g$ in $H^{1/2}(B^3_R)$ as $N\to\infty$. Together with \eqref{eq:dtn_on_harmonics}, which shows that $\Lambda_{v_1,R}g_N = \Lambda_{v_2,R}g_N$, this implies that $\Lambda_{v_1,R}g = \Lambda_{v_2,R}g $ and concludes the proof of \Cref{thm:dtn_extraction}.    
\end{proof}

\subsection{Demonstration of the uniqueness theorem}\label{sec:combining_results_uniqueness}
Now we combine the preliminary results established in \cref{sec:regular_solutions,sec:green_functions,sec:recovering_scattering_matrix,sec:recovering_dtn} to prove \Cref{thm:uniqueness}.

Under the assumptions of \Cref{thm:uniqueness} it follows from \Cref{thm:scattering_matrix_extraction_a} in case (A) and from \Cref{thm:scattering_matrix_extraction_b} in case (B) that $s_{v_1,\ell} = s_{v_2,\ell}$ for all $\ell \geq 0$. Using \Cref{thm:dtn_extraction} we conclude that $\Lambda_{v_1,R} = \Lambda_{v_2,R}$ for any $R$ in the non-empty set $[R_a,\infty) \setminus (\Sigma^D_{v_1,k} \cup \Sigma^D_{v_2,k})$. It follows from the uniqueness theorem of \cite{novikov1988multidimensional}, where the proof does not use that the potential is real-valued, that $v_1 = v_2$ a.e. This proves \Cref{thm:uniqueness}.

\section{Reconstruction}\label{sec:reconstruction}

\subsection{Reconstruction scheme for exact simulated data}\label{sec:algorithm}

The possibility to use measurements of the solar acoustic field at two heigths above the surface to recover the sound speed, density and attenuation inside of the Sun is confirmed by our numerical simulations. In this subsection we shall briefly describe the reconstruction algorithm that we use.

We assume that the unknown solar parameters $q = (c,\rho,\gamma)$ are perturbations of some known background quantities $q^0 =(c^0,\rho^0,\gamma^0)$ such that 
\begin{equation}\label{eq:I_definition}
\supp(q-q^0)\subseteq I, \quad I = [A_1,A_2] \subseteq (0,R_\odot],
\end{equation}
where $R_\odot = 6.957 \times 10^5$ km is the solar radius, and we assume that both parameter sets $q$ and $q^0$ satisfy \eqref{eq:spaces}, \eqref{eq:atmospheric_values}. Let $\Omega \subset \mathbb R_+$ be a finite set of admissible frequencies such that $\Omega \cap (\Sigma^P_q \cup \Sigma^P_{q^0}) = \varnothing$, where 
\begin{equation*}
    \Sigma^P_q = \bigl(0, \tfrac{c_0}{2H}\bigr] \cup \bigl\{ \omega > \tfrac{c_0}{2H} \colon k_\omega \in \Sigma^P_{v_\omega} \},
\end{equation*}
$\Sigma^P_{v_\omega}$ is defined according to \eqref{eq:Sigma_definition}, and the potentials $v_\omega$ and $v^0_{\omega}$ are defined using formula \eqref{eq:acoustic_kv} with parameters $q$ and $q_0$, respectively. We recall that $c_0/(2H)$ is the acoustic cutoff frequency, which separates the regime of oscillations at eigenfrequencies and the scattering regime.

Put $G_q(h,\ell,\omega) = G_{v_\omega,\ell}(R_\odot+h,R_\odot+h)$. As initial data for inversions from exact data we use the imaginary part of the Green's function $\Im G_q(h,\ell,\omega)$ measured at two different non-negative altitudes $h \in \{h_1,h_2\}$, at angular degrees $\ell \in \{0,\dots,\ell_\text{max}\}$, and at all admissible frequencies $\omega \in \Omega$. From this data we recover the solar parameters $q$ as follows.



The first step of the algorithm consists in recovering the scattering matrix elements $s_q(\ell,\omega) = s_{v_\omega,\ell}$, $\ell \in \{0,\dots,\ell_\text{max}\}$, $\omega \in \Omega$, which are defined according to \Cref{thm:regular_solution_expansion}. This reconstruction is done by considering equation \eqref{eq:scattering_matrix_equation} with $r = R_\odot + h$, $h \in \{h_1,h_2\}$ as a linear system for finding $\Re s_q(\ell,\omega)$, $\Im s_q(\ell,\omega)$ at each fixed $\ell$, $\omega$.


At the next step the scattering matrix elements $s_q(\ell,\omega)$ are used to recover the map $u_q \colon I \to \mathbb R^3$ (where $I$ is the interval defined in \eqref{eq:I_definition}) which is defined as follows:
\begin{equation}\label{eq:u_definition}
    \begin{gathered}
    u_q = (u_{q,1},u_{q,2},u_{q,3}), \\
u_{q,1} = \frac{1}{c_0^2}-\frac{1}{c^2}, \quad u_{q,2} = \rho^{\frac 1 2} \left( \frac{d^2}{dr^2} + \frac{2}{r}\frac{d}{dr}\right)\rho^{-\frac 1 2} - \frac{1}{4H^2}, \quad u_{q,3} = \frac{\gamma}{c^2},
  \end{gathered}
\end{equation}
and such that
\begin{equation*}
    \omega^2 u_{q,1} + u_{q,2} - 2i\omega u_{q,3} = v_\omega.
\end{equation*}
This reconstruction is done by applying the iteratively regularized Gauss-Newton method, going back to \cite{bakushinskii1992problem}, to the forward map
\begin{equation}\label{eq:forward_map}
   F \colon L^2(I,\mathbb R^3) \to \mathbb C^{\ell_\text{max}+1} \times \mathbb C^{|\Omega|}, \quad F( u_q ) = s_q.
\end{equation}
For more details on the iteratively regularized Gauss-Newton method and for sufficient conditions of its convergence see, e.g., \cite{kaltenbacher2008iterative}. 



The last step is to determine $q = (c,\rho,\gamma)$ from $u_q$. Note that definitions \eqref{eq:u_definition} lead to the following explicit formulas for determining $c$ and $\gamma$:
\begin{equation*}
    c = \bigl( c_0^{-2} - u_{q,1} \bigr)^{-\frac 1 2}, \quad \gamma = c^2 u_{q,3}.
\end{equation*}
Also note that definitions \eqref{eq:u_definition} lead to the following problem for determining $\rho$:
\begin{equation*}
    -\left( \frac{d^2}{dr^2} + \frac{2}{r}\frac{d}{dr}\right)\rho^{-\frac 1 2} + \left(u_{q,2}+\frac{1}{4H^2}\right) \rho^{-\frac 1 2} = 0, \quad \rho(A_{1,2}) = \rho^0(A_{1,2}),
\end{equation*}
which is solved for the unknown function $\rho^{-\frac 1 2}$. This step concludes the reconstruction algorithm from exact data.

\subsection{Numerical example with exact data}\label{sec:example_exact}

We consider the background sound speed and density from the model of \cite{fournier2017atmospheric}, which extends the standard solar model of \cite{christensen1996current} to the upper atmosphere. We suppose that the background attenuation is equal to $\gamma_0 = 102.5$  $\mu$Hz inside of the Sun and decays to zero smoothly in the region $[R_\odot,R_\odot+h_a]$, where $R_\odot = 6.957 \times 10^5$ km is the solar radius and $h_a = 500$ km is the height above the photosphere at which the (conventional) interface between the lower and upper atmosphere is located. Note that this approximate value for the background attenuation can be obtained by analysing the observed full width at half maximum (FWHM) of acoustic modes, see \cite[section 7.3]{gizon2017computational} for more details\footnote{Our simulations show that another reasonable choice for the attenuation constant in the Sun is $\gamma_0 = 39.8$ $\mu$Hz. For this attenuation coefficient the difference between the observed by HMI and the modeled power spectra (after a linear transformation) is minimal.}. We also assume that the unknown perturbations to the background values of solar parameters are supported in the interval $[0.9 R_\odot,0.95 R_\odot]$.

\begin{figure}[h]
  \insertimg{.31}{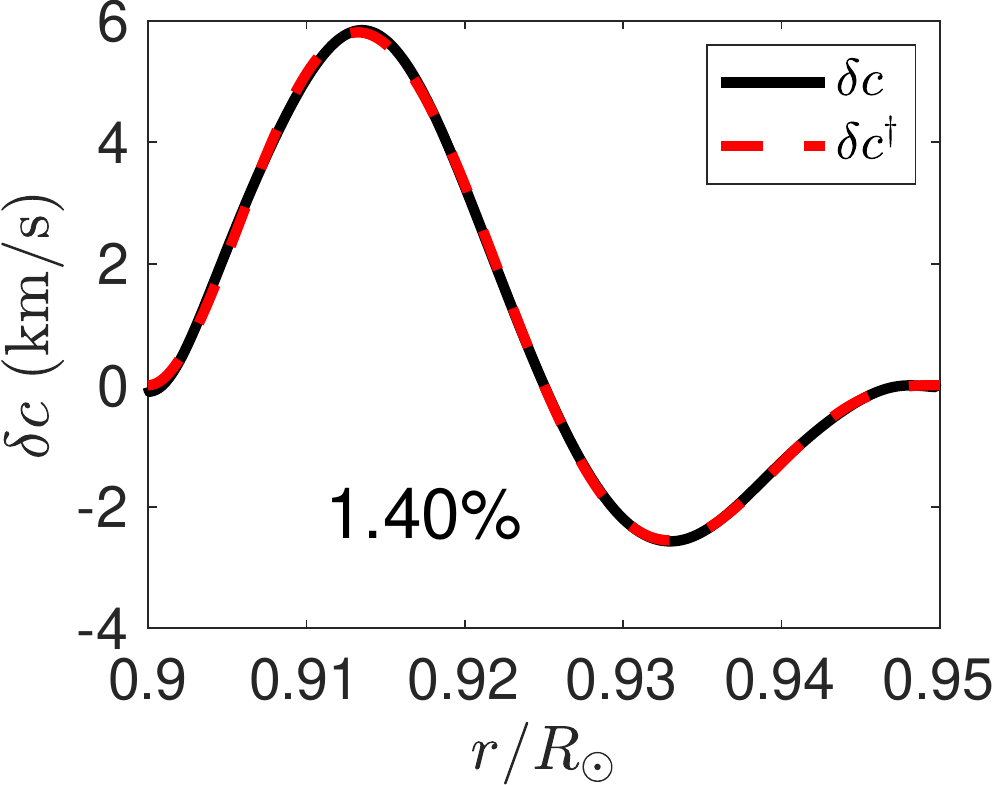}
  \insertimg{.31}{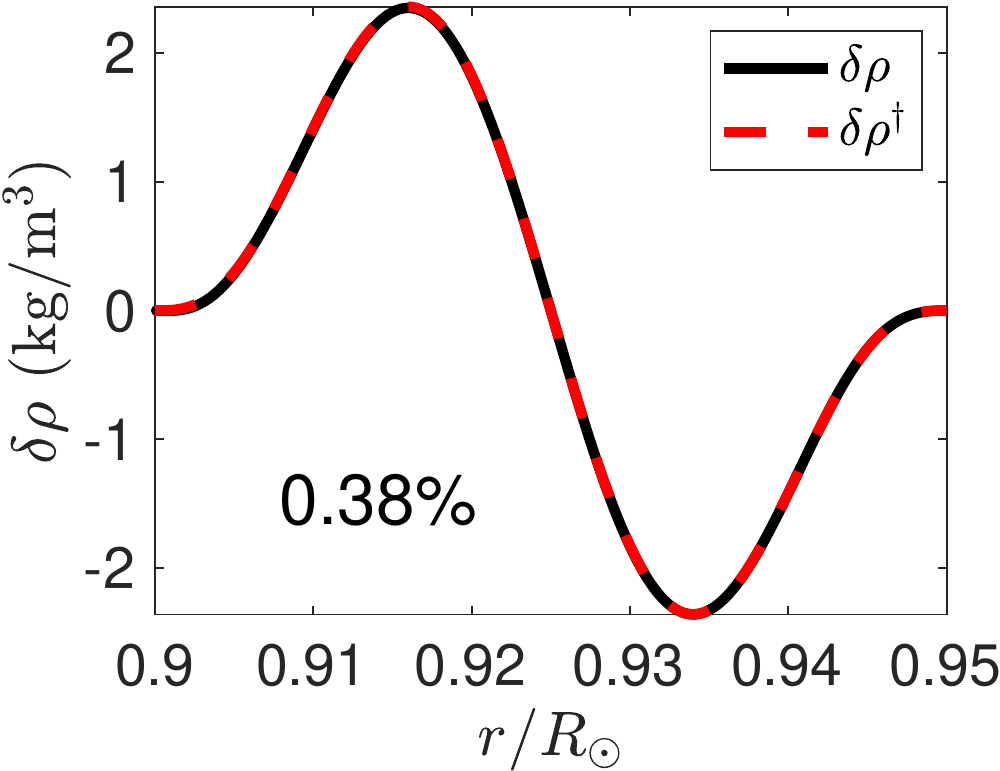}
  \insertimg{.31}{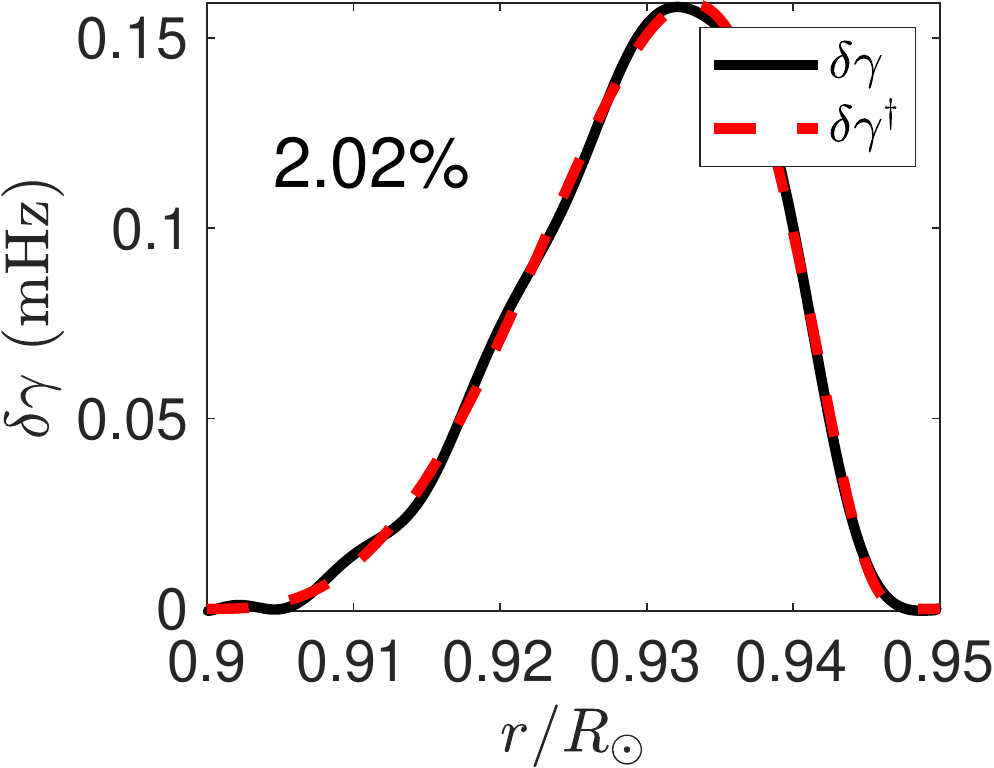}
 \caption{Perturbations $\delta c^\dag$, $\delta \rho^\dag$, $\delta \gamma^\dag$ of solar parameters, reconstructed approximations $\delta c$, $\delta \rho$, $\delta \gamma$ and relative $L^2$ reconstruction errors $e(c,c^\dag)$, $e(\rho,\rho^\dag)$, $e(\gamma,\gamma^\dag)$.}\label{fig.all_i}
\end{figure}

The initial data for reconstructions is the imaginary part of the radiation Green's function $\Im G_q(h,\ell,\omega)$ at heights $h \in \{105, 144\}$ (km), angular degrees $\ell \in \{0, \dots, 250\}$, and frequencies $\omega \in \{5.3, 5.4\}$ (mHz). Note that observations of the acoustic field at these heights approximately correspond to measuring Doppler velocities of the line center (three-point approximation) and center of gravity (six-point approximation) of the Fe I absorption line at 6173 \r{A} using the data from the Helioseismic and Magnetic Imager (HMI) as described in \cite{nagashima2014interpreting}.

\Cref{fig.all_i} shows exact profiles of perturbations $\delta c^\dag$, $\delta \rho^\dag$, $\gamma^\dag$, reconstructed profiles of perturbations $\delta c$, $\delta \rho$, $\delta \gamma$, and relative $L^2$ reconstruction errors $e(\delta c,\delta c^\dag)$, $e(\delta \rho,\delta \rho^\dag)$, $e(\delta \gamma,\delta \gamma^\dag)$. Here,
\begin{equation*}
   e(f,f^\dag) = \frac{\|f-f^\dag\|_{L^2(I)}}{\|f^\dag\|_{L^2(I)}}.
\end{equation*}
This reconstruction example confirms the uniqueness results of \cref{sec:uniqueness}.


\subsection{Reconstruction scheme for noisy data}\label{sec:algorithm_noisy}


The power spectrum $P^m_{v_\omega,\ell}(r) = \mathbb E|\varphi^m_{v_\omega,\ell}(r)|^2$ introduced in \cref{sec:extracting_green_function} can not be measured precisely in a real experiment. A standard approach to compute an approximation to the power spectrum is to parse the time series of acoustic oscillations (after an aproppriate preprocessing) into $N$ segments of equal duration $T$, compute for each segment the sample spectrum (periodogram), and then take the arithmetic average of periodograms $\widehat P^m_{v_\omega,\ell}(r)$. For the definition and properties of the sample power spectrum see, e.g., \cite[section 6.3.1]{jenkins1969spectral}, and for the details on computation of the sample power spectrum from the observed time series of sollar oscillations see \cite{larson2015improved}.

Parameter $T$ is chosen to achieve a desired frequency resolution. We recall that for a time series segment of duration $T$ the frequency resolution of the sample power spectrum is equal to $\Delta \omega = 1/T$. Besides, if the time resolution (cadence) is equal to $\Delta t$ then the sample power spectrum can be computed for frequencies up to $1 / (2\Delta t)$.

It is a standard assumption going back to \cite{woodard1984} that
\begin{equation}\label{eq:chi_square}
   \frac{2N \widehat P^m_{v_\omega,\ell}(r)}{P^m_{v_\omega,\ell}(r)} \; \text{is $\chi^2(2N)$ distributed},
\end{equation}
where $\chi^2(2N)$ denotes the chi-squared distribution with $2N$ degrees of freedom and $r$, $\omega$, $\ell$, $m$ are fixed. We use relation \eqref{eq:chi_square} to simulate data for inversions with noisy data. This simulated data is then used to extract the diagonal values of the imaginary part of the radial Green's function $\Im G_q(h,\ell,\omega) = \Im G_{v_\omega,\ell}(R_\odot+h,R_\odot+h)$ using formula \eqref{eq:green_extraction} with the $O(\tfrac 1 R)$ term dropped and assuming that $\Pi = 1$. Note that in reality $\Pi$ is a function of frequency which is not directly accessible to measurements; for a possible model of this function see, e.g., \cite[section 7.5]{gizon2017computational}. 

Then the reconstruction proceeds as described in \cref{sec:algorithm}. 


\subsection{Numerical examples with noisy data}
We consider a model situation where the solar oscillations are observed for a total period of eight years with the time resolution of 45 s. In this connection, we recall that the HMI observes the sollar oscillations continuously with this cadence since April 30, 2010. We also assume that the sample power spectra are computed from three-day intervals, so that the frequency resolution is equal to 3.86 $\mu$Hz and the maximal frequency is approximately 8.33 mHz.


We simulate $\widehat P_{v_\omega,\ell}^m(R_\odot+h)$ according to \eqref{eq:chi_square} with $N = 974$ (which is the number of three-day intervals constituting eight years of observations) for angular degrees $\ell \in \{0,\dots,250\}$, azimuthal degree $m=0$, observation heights $h \in \{105, 144\}$ (km) and frequencies $\omega \in \{ 5.27,5.29,5.31,5.34,5.36,5.38 \}$ (mHz). We use the same background model and the same assumptions on the unknown parameters as in \cref{sec:example_exact}.



\begin{figure}[h]
   \begin{center}
   \insertimg{.31}{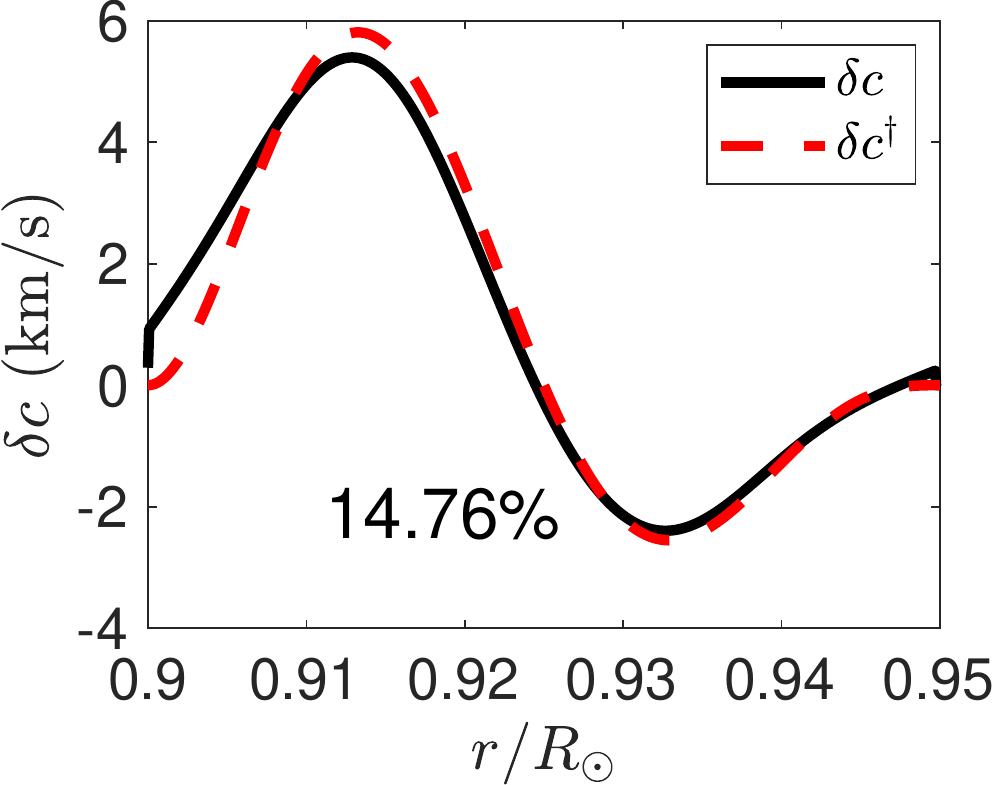}
   \insertimg{.31}{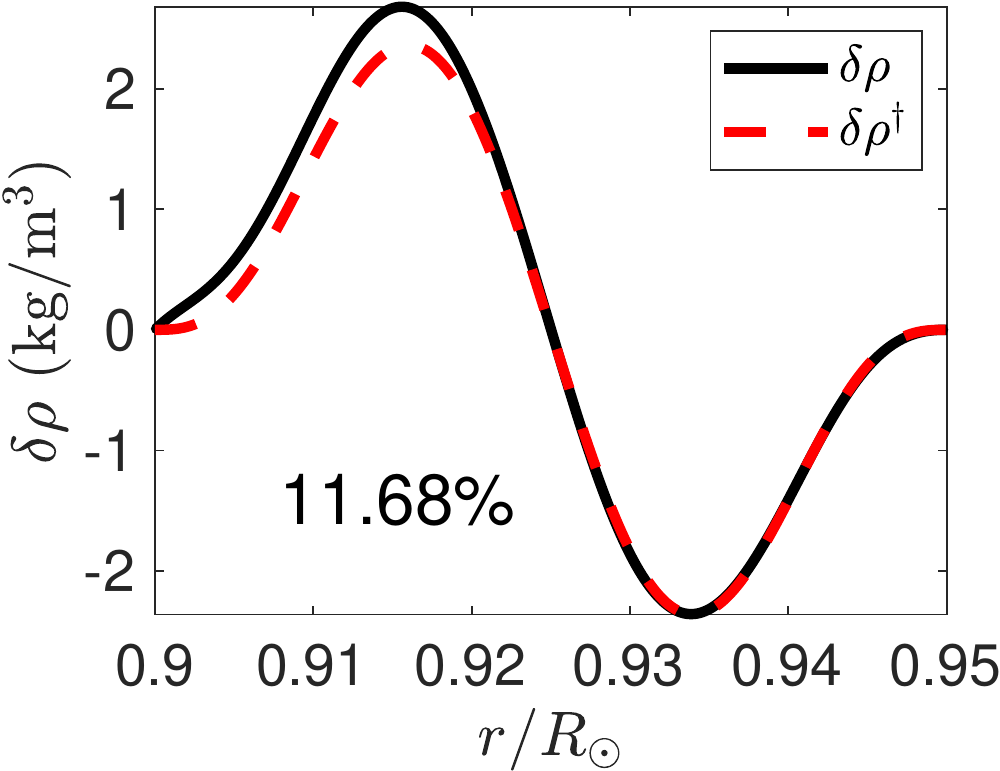}
   \insertimg{.31}{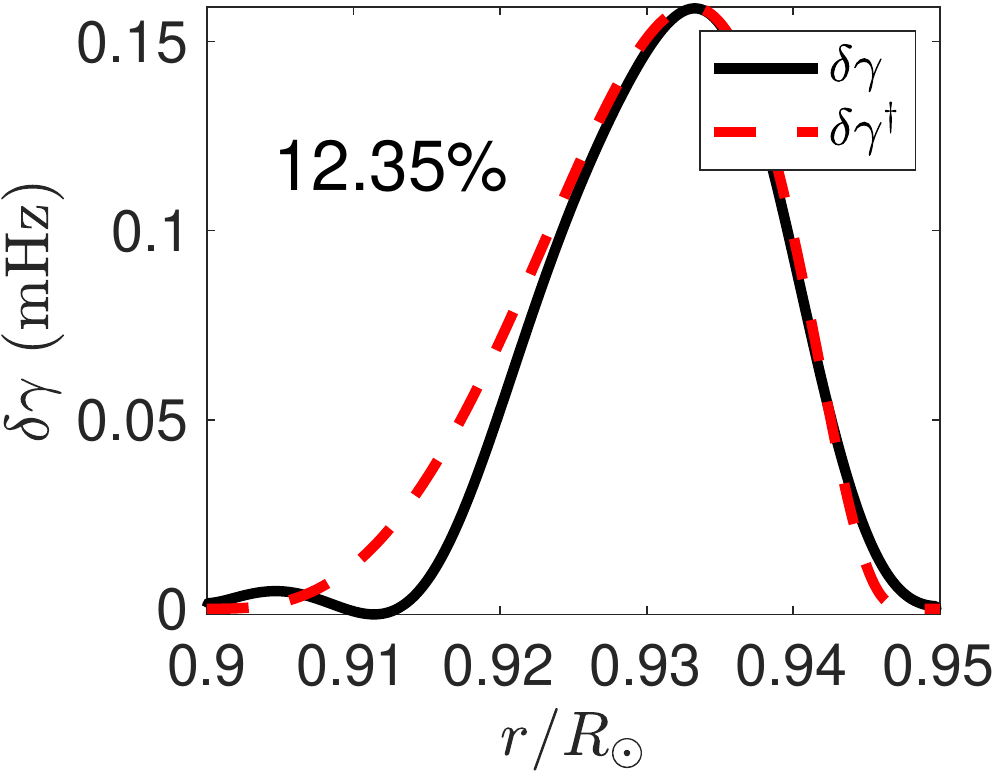}
   \end{center}
   \caption{Perturbations $\delta c^\dag$, $\delta \rho^\dag$, $\delta \gamma^\dag$ of solar parameters, reconstructed approximations $\delta c$, $\delta \rho$, $\delta \gamma$ and relative $L^2$ reconstruction errors for noisy data $e(\delta c,\delta c^\dag)$, $e(\delta \rho,\delta \rho^\dag)$, $e(\delta \gamma,\delta \gamma^\dag)$.}\label{fig.all_8}
\end{figure}

In contrast to reconstructions with exact data, simultaneous recovery of all the parameters from noisy data fails when these realistic settings are used. The point is that the (numerically computed) singular values of the forward map of formula \eqref{eq:forward_map} decrease exponentially fast, which leads to a severe ill-posedness of the inverse problem.

However, if two out of three parameters are known a priori, the third parameter is recovered with reasonable precision. \Cref{fig.all_8} shows parameters $\delta c^\dag$, $\delta \rho^\dag$, $\delta \gamma^\dag$, their reconstructions $\delta c$, $\delta \rho$, $\delta \gamma$ from noisy data and relative $L^2$ reconstruction errors for one realization of data. The mean relative $L^2$ reconstruction errors  for parameters $\delta c^\dag$, $\delta \rho^\dag$, $\delta \gamma^\dag$ are equal to 
\begin{equation*}
e(\delta c, \delta c^\dag) = 11.7\%, \quad e(\delta \rho,\delta \rho^\dag) = 16.8\%, \quad e(\delta \gamma,\delta \gamma^\dag)=11.36\%.
\end{equation*}
Besides, standard deviations of relative $L^2$ reconstruction errors are equal to  3\%, 12.6\%, 3.1\%. We emphasize that in each of these examples two out of three parameters are known a priori and fixed, and we reconstruct the remaining parameter. 

These simulations show that reconstructions from noisy simulated data and, as a corollary, from experimental data, require a separate and detailed treatment, which is out of scope of the present article.


\section{Conclusion} 
We considered the inverse problem of recovering the radially symmetric sound speed, density and attenuation in the Sun from the measurements of the solar acoustic field at two heights above the photosphere and for a finite number of frequencies above the acoustic cutoff frequency. We showed that this problem reduces to recovering a long range potential (with a Coulomb-type decay at infinity) in a Schr\"odinger equation from the measurements of the imaginary part of the radiation Green's function at two distances from zero. We demonstrated that generically this inverse problem for the Schr\"odinger equation admits a unique solution, and that the original inverse problem for the Sun admits a unique solution when measurements are performed at least two different frequencies above the cutoff frequency. These uniqueness results are confirmed by numerical experiments with simulated data without noise. However, simulations also show that the inverse problem is severly ill-posed, and only separate recovery of one of the solar parameters (i.e. when two other parameters are fixed) using a standard iterative reconstruction method (IRGNM) is reasonably precise for realistic noise levels.

\bibliographystyle{plain}
\bibliography{agaltsov2019inverse}

\end{document}